\documentclass{amsart}
\usepackage{amssymb}
\usepackage{amsmath}
\usepackage{latexsym}
\usepackage{amscd}
\usepackage{eufrak}
\usepackage{mathrsfs}
\usepackage[english]{babel}

\newtheorem{theorem}[equation]{Theorem}
\newtheorem{theorem-definition}[equation]{Theorem-Definition}
\newtheorem{lemma-definition}[equation]{Lemma-Definition}
\newtheorem{definition-prop}[equation]{Proposition-Definition}

\newtheorem{prop}[equation]{Proposition}
\newtheorem{lemma}[equation]{Lemma}
\newtheorem{cor}[equation]{Corollary}

\newtheorem{lem}[equation]{Lemma}

\newcommand{\llbracket}{[\negthinspace[}
\newcommand{\rrbracket}{]\negthinspace]}

\newcommand{\llpar}{(\negthinspace(}
\newcommand{\rrpar}{)\negthinspace)}

\theoremstyle{definition}

\newcommand{\N}{\ensuremath{\mathbb{N}}}
\newcommand{\Z}{\ensuremath{\mathbb{Z}}}
\newcommand{\Q}{\ensuremath{\mathbb{Q}}}

\newcommand{\R}{\ensuremath{\mathbb{R}}}
\newcommand{\C}{\ensuremath{\mathbb{C}}}

\newcommand{\cX}{\ensuremath{\mathscr{X}}}

\newcommand{\cC}{\ensuremath{\mathscr{C}}}

\newcommand{\cU}{\ensuremath{\mathscr{U}}}

\newcommand{\cY}{\ensuremath{\mathscr{Y}}}
\newcommand{\cZ}{\ensuremath{\mathscr{Z}}}

\newcommand{\Spec}{\ensuremath{\mathrm{Spec}\,}}
\newcommand{\Spf}{\ensuremath{\mathrm{Spf}\,}}

\newcommand{\red}{\mathrm{red}}

\newcommand{\Gal}{\mathrm{Gal}}

\newcommand{\an}{\mathrm{an}}

\newcommand{\weight}{\mathrm{wt}}

\newcommand{\Sk}{\mathrm{Sk}}
\newcommand{\snc}{\mathrm{snc}}

\newcommand{\D}{\mathcal{D}}

\numberwithin{equation}{subsection}

\newcommand{\sss}{\vspace{5pt} \subsubsection*{ }\refstepcounter{equation}{{\bfseries(\theequation)}\ }}

\begin{document}
\title[The essential skeleton of a degeneration of algebraic varieties]{The essential skeleton of a degeneration of algebraic varieties}

\author{Johannes Nicaise}
\address{KU Leuven\\
Department of Mathematics\\
Celestijnenlaan 200B\\3001 Heverlee \\
Belgium} \email{johannes.nicaise@wis.kuleuven.be}

\author{Chenyang Xu}
\address{Beijing International Center of Mathematics Research\\ Beijing University \\ Beijing \\ China}
\email{cyxu@math.pku.edu.cn}

\begin{abstract}
 In this paper, we explore the connections between the Minimal Model Program and
 the theory of Berkovich spaces.  Let $k$ be a field of characteristic zero and let $X$ be a smooth and projective $k\llpar t\rrpar$-variety with semi-ample canonical
 divisor. We prove that the essential skeleton of
 $X$ coincides with the skeleton of any
minimal $dlt$-model and that it is a strong deformation retract of
the Berkovich analytification of $X$. As an application, we show
that the essential skeleton of a Calabi-Yau variety over $k\llpar
t\rrpar$ is a pseudo-manifold.
%
 \end{abstract}

\maketitle

\tableofcontents

\section{Introduction}
 Let $k$ be a field of characteristic zero and set $R=k\llbracket
 t\rrbracket$ and $K=k\llpar t\rrpar$. We fix a $t$-adic absolute value on $K$ by setting $|t|_K=1/e$. Let $X$ be a geometrically
 connected, smooth and proper $K$-variety. Then one can associate
 to $X$ a $K$-analytic space $X^{\an}$ in the sense of
 \cite{berkbook}. Each point of this space can be interpreted as
 a real valuation on the residue field of a point of $X$,
 extending the $t$-adic valuation on $K$. Thus $X^{\an}$ is
 naturally related to the birational geometry of $R$-models of
 $X$.

  An $snc$-model of $X$ is a regular flat separated $R$-scheme of finite type
  $\cX$, endowed with an isomorphism of $K$-schemes $\cX_K\to X$,
  such that the special fiber $\cX_k$ is a (not necessarily reduced) divisor with strict
  normal crossings. Each $snc$-model $\cX$ of $X$ gives rise to a
  so-called skeleton $\Sk(\cX)$, a finite simplicial space
  embedded in the $K$-analytic space $X^{\an}$, canonically
  homeomorphic to the dual intersection complex $\D(\cX_k)$ of $\cX_k$ \cite[\S3]{MuNi}. If
  $\cX$ is proper over $R$, then $\Sk(\cX)$ is a strong
  deformation retract of $X^{\an}$ (see Theorem \ref{thm:deform} and
  \eqref{sss:defret}). Results of this type are fundamental
  tools in the study of the homotopy type of $K$-analytic spaces,
  for instance in Berkovich's proof of local contractibility of
  smooth $K$-analytic spaces \cite{berk-contr}. On the other hand,
  the fact that the space $X^{\an}$ does not depend on any choice
  of model implies that the homotopy type of $\Sk(\cX)$ does not
  depend on the choice of $\cX$; see \cite{thuillier} for a
  similar result in the context of embedded resolutions of pairs
  of varieties over a perfect field.

If $X$ is a curve of genus $\geq 1$, then it is well-known that
 $X$ has a minimal
 $snc$-model, which gives rise to a {\em canonical} skeleton in $X^{\an}$. However, in higher dimensions, no such
 distinguished $snc$-model exists, and one can wonder if it is
 still possible to construct a canonical skeleton inside the
 space $X^{\an}$. In this paper, we study two such constructions. Although they look quite different at first
 sight, we prove that they indeed yield the same result.

 The first one is the so-called {\em essential skeleton} from \cite[4.6.2]{MuNi}, a
 generalization of a construction of Kontsevich and Soibelman in
 \cite{KoSo} motivated by homological mirror symmetry. Its definition is quite natural:
 for every non-zero regular pluricanonical form $\omega$ on $X$
 and every proper $snc$-model $\cX$ of $X$, the form $\omega$ singles out
 certain faces of the skeleton $\Sk(\cX)$ corresponding to
 intersections of irreducible components where $\omega$ has
 minimal weight in a suitable sense; see \cite[4.5.5]{MuNi} for a precise
 statement. Taking the union of such faces as $\omega$ varies, we
 obtain a simplicial subspace $\Sk(X)$ of $\Sk(\cX)$ that can be
 characterized intrinsically on $X$ and thus no longer depends on
 any choice of an $snc$-model. This space $\Sk(X)$ was called the essential
 skeleton of $X$ in \cite[4.6.2]{MuNi}. If $X$ has trivial
 canonical sheaf, then $\Sk(X)$ coincides with the Kontsevich-Soibelman
 skeleton from \cite{KoSo} associated to any volume form
 $\omega$ on $X$.

  A second construction appears in the context of the Minimal
  Model Program, specifically in the paper \cite{dFKX}. Assume that $X$ is projective. If we enlarge our class of models from
  $snc$-models to so-called $dlt$-models \eqref{sss:dltdef}, then relative minimal models over $\Spec R$
  exist in any dimension, provided that the canonical divisor $K_X$ of the generic fiber is semi-ample
  (see Theorem \ref{thm-mindlt} -- for technical
   reasons, we are obliged to assume that $X$ is defined over an algebraic $k$-curve
    and to work with models over the base curve, because the results from MMP that we use have only been proven for $k$-schemes of finite type). Such a minimal
    $dlt$-model is not unique, but any two of them are crepant
    birational, which implies that their skeleta are the same
    (Corollary \ref{cor-mindlt}). Moreover, we prove that this canonical skeleton is still a
    strong deformation retract of $X^{\an}$ (Corollary
    \ref{cor-sdr}).

     Our main result, Theorem \ref{thm-ess}, states
    that these two constructions are equivalent: if $K_X$ is
    semi-ample, then the essential skeleton $\Sk(X)$ coincides
    with the skeleton of any minimal $dlt$-model.

    We present two applications of this equivalence. First, as an immediate corollary of the above results, we obtain
      that the essential skeleton $\Sk(X)$ is a strong deformation retract of $X^{\an}_K$ when $K_{X}$ is semi-ample (see Corollary \ref{cor-defret}).
   Second, in Section
    \ref{ss-CY}, we study the topological
    properties of the essential skeleton of a Calabi-Yau variety
    $X$
    over $K$. Using \cite{KK,Kol11}, we show that $\Sk(X)$ is a pseudo-manifold with boundary, and even a closed pseudo-manifold
     when $\Sk(X)$ has maximal dimension and $k$ is algebraically closed. Moreover,
      using logarithmic geometry, we show that $\Sk(X)$ only depends on the reduction modulo $t^2$ of any
    proper $snc$-model of $X$, which allows us to remove the
    technical assumption that $X$ is defined over a curve (Theorem \ref{thm:CY}).

\addtocontents{toc}{\protect\setcounter{tocdepth}{-1}}

\subsection*{Acknowledgements} We are grateful to Tommaso de
Fernex and J\'anos Koll\'ar for helpful discussions, and to the referee for carefully reading the manuscript
and making valuable suggestions. This joint
work was started when both of the authors attended the conference
{\it Arithmetic Algebraic Geometry} held in Berlin in
 June 2013. We thank the organizers, especially H\'el\`ene
Esnault, for the hospitality.  JN is partially
supported by the ERC Starting Grant MOTZETA (project 306610) of the European Research Council.  CX is partially supported by the
grant `Recruitment Program of Global Experts'.

\subsection*{Terminology and conventions}
 We follow \cite{Kol13} for the definitions of various notions of a singular pair from the Minimal Model Program, including {\it  klt, dlt} and {\em log canonical pairs}.
  In particular, we refer to \cite[4.15]{Kol13} for the definition of {\em log canonical
  centers}. The non-archimedean analytic spaces that appear in
  this paper are $K$-analytic spaces in the sense of
  \cite{berkbook}. We refer to \cite{temkin} for a gentle
  introduction. We will also make use of some basic logarithmic
  geometry; all log structures in this paper are defined with respect
  to the Zariski topology, and they are fine and saturated ($fs$).
  The standard introduction to logarithmic geometry is
  \cite{kato-intro}.

\addtocontents{toc}{\protect\setcounter{tocdepth}{2}}

\section{Minimal $dlt$-models}

\subsection{A few reminders on Berkovich spaces and formal models}
\sss Let $k$ be a field of characteristic zero. We set
$R=k\llbracket t\rrbracket$ and $K=k\llpar t\rrpar $, and we fix a
$t$-adic absolute value $|\cdot|_K$ on $K$ by setting $|t|_K=1/e$.
For every $K$-scheme of finite type $Y$, we denote by $Y^{\an}$
the associated $K$-analytic space. As a set, it consists of the couples of the form $x=(y,|\cdot|)$ where $y$ is a point
of $Y$ and $|\cdot|$ is an absolute value on the residue field
 $\kappa(y)$ of $Y$ at $y$ that extends the absolute value $|\cdot|_K$ on $K$. The topology on $Y^{\an}$ is the weakest topology such that the forgetful map $$\iota:Y^{\an}\to Y:(y,|\cdot|)\mapsto y$$ is continuous and such that, for every open subset $U$ of $Y$ and every regular function $f$ on $U$, the map
 $$\iota^{-1}(U)\to \R_{\geq 0}:(y,|\cdot|)\mapsto |f(y)|$$ is continuous.  The residue field of $Y^{\an}$
 at a point $x=(y,|\cdot|)$ is the completion of $\kappa(y)$ with respect to the absolute value $|\cdot|$. It is denoted by $\mathscr{H}(x)$, and we will write $\mathscr{H}(x)^o$ for its valuation ring.

\sss For every separated $R$-scheme
of finite type $\cY$ we set $\cY_k=\cY\times_R k$ and
$\cY_K=\cY\times_R K$. We say that a point $x$ of $(\cY_K)^{\an}$ has a center on $\cY$ if
 the natural morphism $$\Spec \mathscr{H}(x)\to \cY_K$$ extends to a morphism
$$\Spec \mathscr{H}(x)^o\to \cY.$$ Such an extension is necessarily unique by the valuative criterion for separatedness. If it exists, the image of the closed point of $\Spec \mathscr{H}(x)^o$ is a point of $\cY_k$ that we call the center of $x$ on $\cY$ and that we denote by $\red_{\cY}(x)$. The set of all the points on $(\cY_K)^{\an}$ with a center
 on $\cY$ is denoted by $\widehat{\cY}_\eta$, and the map
  $$\red_{\cY}:\widehat{\cY}_\eta\to \cY_k:x\mapsto \red_{\cY}(x)$$ is called the reduction
map. It is anti-continuous, meaning that the inverse image of every open is closed. Our notation is justified by the following fact:
 if we denote by $\widehat{\cY}$
the formal $t$-adic completion of $\cY$, then $\widehat{\cY}_\eta$
 is a compact analytic domain in $\cY_K^{\an}$ that is canonically isomorphic to the generic fiber of
 $\widehat{\cY}$ in the category of $K$-analytic spaces.
 If $\cY$ is proper over $R$, then
 $\widehat{\cY}_\eta=\cY_K^{\an}$ by the valuative criterion for properness.

\subsection{Models and log pullbacks}
\sss Let $\cC$ be a connected smooth algebraic curve over $k$. Let
$s$ be a $k$-rational point on $\cC$ and set $C=\cC\setminus
\{s\}$. We fix a uniformizer $t$ in $\mathcal{O}_{\cC,s}$. This
choice determines an isomorphism of $k$-algebras $R\to
\widehat{\mathcal{O}}_{\cC,s}$ and thus a morphism of $k$-schemes
$\Spec R\to \cC$.

 \sss \label{sss:not} Let
$X$ be a smooth and proper scheme over $C$ with geometrically
connected fibers. A model of $X$ over $\cC$ is a flat separated
 $\cC$-scheme of finite type $\cX$ endowed with an isomorphism of $C$-schemes
$\cX\times_{\cC} C\to X$. Note that we do not require $\cX$ to be
proper over $\cC$. Morphisms of models are defined in the usual
way. We denote by $\cX_s$ the fiber of $\cX$ over $s$, by
 $\cX_R$ the base change of $\cX$ to $\Spec R$ and by $X_K$ the
base change of $X$ to $\Spec K$. We denote by $K_X$ a relative
canonical divisor for $X$ over $C$, and for every normal model
$\cX$ of $X$, we denote by $K_{\cX}$ a relative canonical divisor
for $\cX$ over $\cC$.

\sss For every $\cC$-model $\cX$ of $X$, we denote by $\cX^{\snc}$
the subset of $\cX$ consisting of the points where $\cX$ is
regular and $\cX_s$ is a divisor with strict normal crossings
(some authors use the terminology ``simple normal crossings''
instead). Thus $\cX^{\snc}$ is the union of $X$ with the set of
points $x$ of $\cX_s$ such that $\mathcal{O}_{\cX,x}$ is regular
and there exist a unit $u$ and a regular system of local
parameters $(z_1,\ldots,z_n)$ in $\mathcal{O}_{\cX,x}$ and
non-negative integers $N_1,\ldots,N_n$ such that
$$t=u\prod_{i=1}^n (z_i)^{N_i}.$$  The subset $\cX^{\snc}$ is an
open subscheme of $\cX$ and it is again a $\cC$-model of $X$.
Moreover, if $\cX$ is normal, then $\cX^{\snc}_s$ is dense in
$\cX_s$.  We say that $\cX$ is an $snc$-model of $X$ if
$\cX=\cX^{\snc}$, that is, if $\cX$ is regular and $\cX_s$ is a
divisor with strict normal crossings.  If $\cX$ is a model of $X$
over $\cC$, then a log resolution of $(\cX,\cX_s)$ is a proper
morphism of $\cC$-models $h:\cY\to \cX$ such that $\cY$ is an
$snc$-model of $X$.

\sss \label{sss-logpb} Let $h:\cY\to \cX$ be a proper morphism of
normal
 $\cC$-models of $X$. Assume that $K_{\cX}+(\cX_s)_{\red}$
is $\Q$-Cartier. Then the log pullback of $(\cX_s)_{\red}$ to
$\cY$ is the unique $\Q$-Weil divisor $\Delta$ on $\cY$ such that
$K_{\cY}+\Delta$ is $\Q$-linearly equivalent to
$$f^*(K_{\cX}+(\cX_s)_{\red})$$ and $f_*\Delta=(\cX_s)_{\red}$.

\sss \label{sss:notMuNi} We will use the following notations from
\cite{MuNi}. If $\cX$ is a normal model of $X$ over $\cC$, $x$ is
a point of $\widehat{\cX}_\eta$ and $D$ is a divisor on $\cX$ that
is supported on $\cX_s$ and Cartier at $\red_{\cX}(x)$, then we
set
$$v_x(D)=-\ln |f(x)|$$ where $f$ is any element of the local ring
of $\cX$ at $\red_{\cX}(x)$ such that $D=\mathrm{div}(f)$ locally
at $\red_{\cX}(x)$. It is clear that $v_x(D)$ is linear in $D$. If
$\cX$ is regular and $\omega$ is a non-zero rational section of
$\omega_{\cX_R/R}^{\otimes m}$, for some $m>0$ (for instance, an
 $m$-pluricanonical form on $X_K$) then we denote by
 $\mathrm{div}_{\cX}(\omega)$ the corresponding divisor on $\cX_R$.

\subsection{$dlt$-models}
\sss \label{sss:dltdef}  A $dlt$-model of $X$ is a normal proper
$\cC$-model $\cX$ of $X$
 such that $(\cX,(\cX_s)_{\red})$ is a $dlt$-pair. This means that
 $(\cX,(\cX_s)_{\red})$ is log canonical and that each log
 canonical center of $(\cX,(\cX_s)_{\red})$ has non-empty
 intersection with $\cX^{\snc}$.  In
 particular, every proper $snc$-model of $X$ is a $dlt$-model. An equivalent formulation of
 the
 definition is the following: $K_{\cX}+(\cX_s)_{\red}$ is
 $\Q$-Cartier, and for every log resolution $h:\cY\to \cX$ of
 $(\cX,\cX_s)$ and every irreducible component $E$ of $\cY_s$, the
 multiplicity of $E$ in the log pullback $\Delta$ of
 $(\cX_s)_{\red}$ to $\cY$ is at most $1$. Moreover, if it is
 equal to $1$, then $h(E)$ must have non-empty intersection with
 $\cX^{\snc}$.
  In practice, we will apply the $dlt$
property {\em via} Lemma \ref{lemm:dlt} below.

\sss  We say that a $dlt$-model $\cX$ of $X$ is a {\it good minimal
model} if $\cX$ is
  $\Q$-factorial and
$K_{\cX}+(\cX_s)_{\red}$ is semi-ample over $\cC$.

\sss For every $dlt$-model $\cX$ of $X$, we can define the dual
complex $\D((\cX_s)_{\red})$ for the $dlt$-pair
$(\cX,(\cX_s)_{\red})$ by gluing cells corresponding to
irreducible components of intersections of irreducible components
of $\cX_s$, as in Definition 8 in \cite{dFKX}. When $k$ is not
algebraically closed, we note that we only glue cells
corresponding to irreducible components (instead of geometrically
irreducible components). In other words, $\D((\cX_s)_{\red})$ is
the quotient of the $\Gal(k^a/k)$-equivariant dual complex
constructed in \cite[\S31]{dFKX}.

\sss For every $dlt$-model $\cX$ of $X$, the log canonical centers
of $(\cX,(\cX_s)_{\red})$ are the irreducible components of
intersections of irreducible components of $(\cX_s)_{\red}$, by
\cite[4.16]{Kol13}. These are also precisely the closures
  in $\cX_s$ of the connected components
 of intersections of irreducible components of $\cX^{\snc}_s$ (since these connected components are the log canonical centers of
 $(\cX^{\snc},(\cX_s^{\snc})_{\red})$). Thus, the dual intersection complex $\D((\cX_s)_{\red})$ is the same
 as the dual intersection complex of the strict normal crossings
 divisor
 $\cX^{\snc}_s$, and the cells of this complex correspond bijectively to the log canonical centers
 of $(\cX,(\cX_s)_{\red})$. See Section 2 in \cite{dFKX} for more background.

\sss Let $\cX_1$ and $\cX_2$ be two $dlt$-models of $X$ over
$\cC$. We say that $\cX_1$ and $\cX_2$ are crepant birational if
there exist a normal proper $\cC$-model $\cY$ of $X$ and morphisms
of $\cC$-models $f_i:\cY\to \cX_i$ for $i=1,2$ such that the log
pullbacks of $(\cX_{1,s})_{\red}$ and $(\cX_{2,s})_{\red}$
coincide (see \cite[2.23]{Kol13}). Note that we can always assume that $\cY$ is an
$snc$-model, by taking a log resolution of $(\cY,\cY_s)$.
 The following theorem collects two fundamental results from
the Minimal Model Program.

\begin{theorem}\label{thm-mindlt}\item
\begin{enumerate}
\item \label{item:exist} If  $X$ is projective over $C$, then $X$
has a good minimal $dlt$-model if and only if $K_X$ is semi-ample
over $C$.

\item \label{item:crep} Any two good minimal $dlt$-models of $X$
are crepant birational.
\end{enumerate}
\end{theorem}
\begin{proof}
\eqref{item:exist} The condition that $K_X$ is semi-ample over $C$
is obviously necessary, since for every $dlt$-model $\cX$ of $X$,
the divisor $K_X$ is $\Q$-linearly equivalent to the restriction
of $K_{\cX}+(\cX_s)_{\red}$ to $X$. Conversely, assume that $K_X$
is semi-ample over $C$, and let $\cY$ be a projective $snc$-model
of $X$. Then
 applying \cite[2.12]{HX} to the $dlt$-pair
 $(\cY,(\cY_s)_{\red})$, we see that $X$ has a good minimal
 $dlt$-model. Condition (1) of \cite[2.12]{HX} follows from our assumption, and condition (2)
   follows from the following observation. Let $m$ be a positive integer such that
  $m(K_{\cY}+(\cY_s)_{\red})$ is Cartier. Over a sufficiently small open neighbourhood of $s$ in $\cC$, we have an isomorphism of $\mathcal{O}_{\cC}$-algebras
 $$R(\cY/\cC,m(K_{\cY}+(\cY_s)_{\red}))\cong R(\cY/\cC,m(K_{\cY}+(\cY_s)_{\red})-\cY_s),$$
 where  $ R(\cY/\cC,L):=\bigoplus_{j\geq 0} \pi_*(\mathcal{O}_{\cY}(jL))$ with $\pi:\cY\to \cC$ the structural morphism.
 Thus it suffices to show that
 $$\mathcal{A}=R(\cY/\cC,m(K_{\cY}+(\cY_s)_{\red})-\cY_s)$$ is a finitely generated $\mathcal{O}_{\cC}$-algebra.
   If we denote by $M$ the maximum of
 the multiplicities of the components in $\cY_s$ then we may assume that $m> M$,
 so that
$$(\cY,(\cY_s)_{\red}-\frac{1}{m}\cY_{s})$$ is an effective $klt$ pair.
 Hence, the finite generation of $\mathcal{A}$ follows from \cite{BCHM}.

\eqref{item:crep} It is already observed in Definition 15 of
\cite{dFKX} that this follows from the proof of \cite[3.52]{KM}.
\end{proof}


\section{The essential skeleton}
\subsection{Retraction to the skeleton of an $snc$-model}
\sss Let $\cY$ be a connected regular flat separated $R$-scheme of
finite type such that the special fiber $\cY_k$ is a divisor with
strict normal crossings. Then, as explained in \cite[\S3.1]{MuNi},
one can associate to $\cY$ its skeleton $\Sk(\cY)$, which is a
topological subspace of the generic fiber $\widehat{\cY}_\eta$ of
 the formal $t$-adic completion of $\cY$. It is the set of points
 of $\widehat{\cY}_\eta$ that correspond to a real valuation
 on the function field of $\cY_K$ that is monomial with
 respect to the strict normal crossings divisor $\cY_k$.
 The skeleton $\Sk(\cY)$ is canonically homeomorphic to the
  dual intersection complex $\D((\cY_k)_{\red})$ of
 $\cY_k$. There exists a canonical continuous retraction
 $$\rho_{\cY}:\widehat{\cY}_\eta\to \Sk(\cY)$$
  which maps each point $x$ of $\widehat{\cY}_\eta$ to the unique point $x'$
  of $\Sk(\cY)$ such that $\red_{\cY}(x)$ is contained in the closure of $\red_{\cY}(x')$ and such that
  $|f(x)|=|f(x')|$ for every element $f$ in $\mathcal{O}_{\cY,\red_{\cY}(x)}$ that is invertible on the generic fiber $\cY_K$.
   Thus one can view $x'=\rho_{\cY}(x)$ as the ``best" approximation of $x$ by a monomial valuation on the pair $(\cY,\cY_k)$.
   Moreover,
 $\Sk(\cY)$ carries a canonical piecewise $\Z$-affine structure
 \cite[\S3.2]{MuNi}.

\sss We keep the notations from \eqref{sss:not}. For each
$\cC$-model $\cX$ of $X$, we define the skeleton of $\cX$ by
$$\Sk(\cX)=\Sk(\cX^{\snc}_R)\subset
\widehat{(\cX^{\snc}_R)}_\eta\subset X_K^{\an}$$ and we write
$\rho_{\cX}$ for $\rho_{\cX_R^{\snc}}$.
 If $\cX$ is a proper $snc$-model of $X$, one has the following
 crucial property.

\begin{theorem}\label{thm:deform}
If $\cX$ is a proper $snc$-model of $X$ over $\cC$, then there
exists a
 continuous map
$$H:[0,1]\times X_K^{\an}\to X_K^{\an}$$ such that
$H(0,\cdot)$ is the identity, $H(t,x)=x$ for all $x$ in $\Sk(\cX)$
and all $t$ in $[0,1]$, and $H(1,\cdot)=\rho_{\cX}$. Thus
$\Sk(\cX)$ is a strong deformation retract of $X_K^{\an}$.
\end{theorem}
\begin{proof}
 A closely related result is proven in \cite[3.26]{thuillier}. We
 will explain how our statement can be deduced from that result.
 Following the notation in \cite{thuillier}, we denote by $\cX^{\beth}$ the $k$-analytic space associated to the
 toroidal embedding $X\hookrightarrow \cX$, where $k$ is endowed
 with the trivial absolute value. By definition,
 $\cX^{\beth}$ is the generic fiber of the formal $t$-adic
 completion $\widehat{\cX}$ of $\cX$, viewed as a special formal $k$-scheme by
 forgetting the $k\llbracket t\rrbracket $-structure \cite[\S1]{berk}.

 The relation between $\cX^{\beth}$ and $X_K^{\an}$ is explained in detail at the beginning
  of Section 4 in \cite{Ni-sing}; let us recall the main idea. Considering
 the morphism of special formal $k$-schemes $\widehat{\cX}\to \Spf
 k\llbracket t\rrbracket $ and passing to the generic fibers, we obtain a morphism
 of $k$-analytic spaces from $\cX^{\beth}$ to the open unit disc
 $D$ over $k$. We can identify the underlying topological space of $D$ with $[0,1[$
 by means of the homeomorphism
 $$D\to [0,1[\,:x\mapsto |t(x)|.$$ The residue field of $D$ at the point $1/e$ in $[0,1[$ is $K$ with
  our chosen $t$-adic absolute value $|\cdot|_K$, and the $K$-analytic
  space $X_K^{\an}$ is canonically isomorphic to the fiber of
  $\cX^{\beth}$ over $1/e$. Thus we can view $X_K^{\an}$ as the
  subspace of $\cX^{\beth}$ consisting of the points $x$ such that
  $|t(x)|=1/e$.

   In \cite[3.13]{thuillier}, Thuillier constructs a retraction
   $p_{\cX}$ of $\cX^{\beth}$ onto a certain subspace
   $\mathcal{S}(\cX)$, the skeleton of the toroidal embedding.
   Moreover, in \cite[3.26]{thuillier}, he shows that $p_{\cX}$
   can be extended to a strong deformation retraction $H$ of
   $\cX^{\beth}$ onto $\mathcal{S}(\cX)$. Going through the
   definitions, one observes that $p_{\cX}$ and $H$ commute with
   the morphism $\cX^{\beth}\to D$ and that the restriction of
   $$p_{\cX}:\cX^{\beth}\to \mathcal{S}(\cX)$$ over the point
   $1/e$ of $D$ is precisely the retraction
   $$\rho_{\cX}:X_K^{\an}\to \Sk(\cX).$$ Thus by restricting $H$
   over $1/e\in D$, we obtain a map that satisfies all the
   properties in the statement.
\end{proof}

\sss \label{sss:defret} Theorem \ref{thm:deform} can be extended
to the case where $X$ is defined over $K$ instead of $C$ and $\cX$
is a proper $snc$-model of $X$ over $R$. The general proof
technique is the same as in \cite{thuillier}, but one replaces the
formalism of toroidal embeddings by the more flexible language of
logarithmic geometry. Since the proof requires some technical
 preparations, we will present the details in a separate paper
\cite{Ni-log}. We will only use this generalization in the proof
of Theorem \ref{thm-form}.

\subsection{The skeleton of a good minimal $dlt$-model}
\sss In the following subsections, we will make use of the {\em
weight function}
 $$\weight_{\omega}:X_K^{\an}\to \R\cup \{+\infty\}$$ associated
 to a non-zero $m$-pluricanonical form on $X_K$, for any $m>0$.
 Its construction and main properties are described
 in \cite[4.4.5]{MuNi}. For us, its most important features are
 the following: if $\cX$ is an $snc$-model of $X$ over $\cC$ and $x$ is a
 point of $\Sk(\cX)$, then
 $$\weight_{\omega}(x)=v_x(\mathrm{div}_{\cX}(\omega)+m(\cX_s)_{\red})$$
 (here we use the notation recalled in \eqref{sss:notMuNi}). Moreover, for every point $y$ of $\widehat{\cX}_\eta$, we have
 $$\weight_{\omega}(y)\geq \weight_{\omega}(\rho_{\cX}(y))$$ with
 equality if and only if $y$ lies on $\Sk(\cX)$.

  \sss \label{sss:log} It will often be useful to interpret the weight function in
  terms of logarithmic differential forms. Let $\cY$ be a regular
  separated $R$-scheme of finite type such that $\cY_k$ is a
  divisor with strict normal crossings. We write $S^+$ for the log
  scheme associated to $R\setminus \{0\}\to R$ and $\cY^+$ for the
  log scheme obtained by endowing $\cY$ with the divisorial log
  structure associated to $\cY_k$. Then $\cY^+$ is log smooth over
  $S^+$. If we denote by $j:\cY_K\to \cY$ the natural open
  immersion, then a simple computation shows that the sub-$\mathcal{O}_{\cY}$-module $\omega_{\cY^+/S^+}$ of
 $j_*\omega_{\cY_K/K}$ is equal to $\omega_{\cY/R}((\cY_k)_{\red}-\cY_k)$
 (it suffices to check that these line bundles coincide at the generic
 points of the special fiber $\cY_k$). Thus if $\omega$ is an
 $m$-pluricanonical form on $X_K$ and $\cX$ is an $snc$-model of
 $X$ over $\cC$, then
 $$\weight_{\omega}(x)=v_x(\mathrm{div}_{\cX^+}(\omega))+m$$
 for every point $x$ of $\Sk(\cX)$, where we denote by
 $\mathrm{div}_{\cX^+}(\omega)$ the divisor on $\cX_R$ associated
 to $\omega$ viewed as a rational section of the line bundle
 $\omega^{\otimes m}_{\cX^+_R/S^+}$.

\begin{lemma}\label{lemm:dlt}  Let $\cX$ be a $dlt$-model of $X$
and let $h:\cY\to \cX$ be a log resolution of $(\cX,\cX_s)$.
Denote by $\Delta$ the log pullback of $(\cX_s)_{\red}$ to $\cY$.
Let $x$ be a point of $\Sk(\cY)$ such that $\red_{\cX}(x)$ does
not lie in $\cX^{\snc}$. Then $\Delta<(\cY_s)_{\red}$ locally at
$\red_{\cY}(x)$.
\end{lemma}
\begin{proof}
By the definition of a $dlt$-model, we know that $\Delta\leq
(\cY_s)_{\red}$. Thus it suffices to show that these divisors are
different locally at $\red_{\cY}(x)$. Since $x$ lies on
$\Sk(\cY)$,
 its reduction $\red_{\cY}(x)$ is a generic point of the intersection of the
irreducible components of $\cY_s$ that contain $\red_{\cY}(x)$.
Thus if we denote by $h':\cY'\to \cY$ the blow-up of $\cY$ at the
closure of $\red_{\cY}(x)$, then $\cY'$ is again an
 $snc$-model of $X$.

  We denote by $\Delta'$ the log pullback of
$\Delta$ to $\cY'$. The image of the exceptional divisor $E$ of
$h'$ in $\cX$ is the closure of $\red_{\cX}(x)=h(\red_{\cY}(x))$
and thus disjoint from $\cX^{\snc}$. By the definition of a
$dlt$-model, we know that the multiplicity of $E$ in $\Delta'$ is
strictly smaller
 than $1$. Since the
log pullback of $(\cY_s)_{\red}$ to $\cY'$ is equal to
 $(\cY'_s)_{\red}$, we see that $\Delta<(\cY_s)_{\red}$ locally at
$\red_{\cY}(x)$.
\end{proof}

\begin{prop}\label{prop-min}
Let $\cX$ be a $dlt$-model of $X$ over $\cC$, let $\cY$ be a
proper $snc$-model of $X$ over $\cC$ and let $h:\cY\to \cX$ be a
morphism of $\cC$-models. Denote by $\Delta$ the log pullback of
$(\cX_s)_{\red}$ to $\cY$. If we set
$$S=\{x\in \Sk(\cY)\,|\,v_x(\Delta)=v_x((\cY_s)_{\red}) \}$$ then
$\Sk(\cX)=S$.
\end{prop}
\begin{proof} Applying  \cite[3.1.7]{MuNi} to the proper morphism $h^{-1}(\cX^{\snc})\to \cX^{\snc}$, we see that $\Sk(\cX)$
is contained in $\Sk(\cY)$. Moreover, it follows from
 Lemma \ref{lemm:dlt} that for every point $x$ of $S$, the reduction
 $\red_{\cX}(x)$ must be contained in
$\cX^{\snc}$.
  Now let $x$ be any point in $\Sk(\cY)$ such that $\red_{\cX}(x)$ lies in $\cX^{\snc}$.
 We must show that $v_x(\Delta)=v_x((\cY_s)_{\red})$ if and only if $x$ lies in $\Sk(\cX)$, or, equivalently, $x$ is equal to its projection
 $$x'=\rho_{\cX}(x)$$ to the skeleton of $\cX$. Let $\omega$ be a local generator of $\omega_{\cX^{\snc}/\cC}$ at $\red_{\cX}(x)$.
 It induces a rational section of the canonical bundle
 $\omega_{X_K/K}$ by base change.
 By
 \cite[4.4.5]{MuNi}, we know that $x=x'$ if and only if
 $$\weight_{\omega}(x)= \weight_{\omega}(x').$$
  Since
  the divisor of $\omega$ is zero in a neighbourhood of
  $\red_{\cX}(x')$, we have
  $$\weight_{\omega}(x')=v_{x'}((\cX_s)_{\red})=v_{x}((\cX_s)_{\red}).$$
   On the other hand, computing $\weight_{\omega}(x)$ on the model $\cY$ we get
   $$\weight_{\omega}(x)=v_x(\mathrm{div}_{\cY}(\omega)+(\cY_s)_{\red})=v_x((\cX_s)_{\red})+v_x((\cY_s)_{\red}-\Delta).$$
     Thus we see that $\Sk(\cX)=S$.
\end{proof}

\begin{cor}\label{cor-crep}
Let $\cX_1$ and $\cX_2$ be two $dlt$-models of $X$ over $\cC$. If
$\cX_1$ and $\cX_2$ are crepant birational, then
$\Sk(\cX_1)=\Sk(\cX_2)$.
\end{cor}
\begin{proof}
This follows immediately from Proposition \ref{prop-min}.
\end{proof}

\sss Corollary \ref{cor-crep} implies, in particular, that the
 skeleta $\Sk(\cX_1)=\Sk(\cX_2)$ are isomorphic as topological
spaces with piecewise affine structure, by \cite[\S3.2]{MuNi}.
Since $\Sk(\cX_i)$ is canonically homeomorphic to the dual complex
associated to the reduced special fiber of $\cX_i$,  for $i=1,2$,
this also follows from Proposition 11 in \cite{dFKX}, whose proof
relies on Weak Factorization. The proofs of Corollary
\ref{cor-crep} and \cite[\S3.2]{MuNi} do not use Weak
Factorization.

\begin{cor}\label{cor-mindlt}
 If $K_X$ is semi-ample, then the skeleton of a good minimal
 $dlt$-model of $X$ does not depend on the choice of the good minimal
 $dlt$-model.
\end{cor}
\begin{proof}
This follows from Theorem \ref{thm-mindlt}\eqref{item:crep} and
Corollary \ref{cor-crep}.
\end{proof}

\begin{theorem}\label{thm-collapse}
Assume that $X$ is projective over $C$ and $K_X$ is semi-ample
over $C$. If $\cX$ is a good minimal $dlt$-model of $X$ and $\cY$
is any projective $dlt$-model of $X$, then $\Sk(\cX)$ is contained
in $\Sk(\cY)$. Moreover, $\Sk(\cX)$ can be obtained from
$\Sk(\cY)$ (as a topological subspace of $\Sk(\cY)$ with piecewise
affine structure) by a finite number of elementary collapses.
\end{theorem}
\begin{proof}
For the definition of an elementary collapse in a simplicial
topological space, we refer to Definition 18 in \cite{dFKX}. By
 Corollary
\ref{cor-mindlt}, we can assume that the good minimal $dlt$-model
$\cX$ is the result of running MMP for $(\cY,(\cY_s)_{\red})$. Now
the statement follows from Corollary 22 in \cite{dFKX}. When $k$
is not algebraically closed, see also \S31 in \cite{dFKX}.
\end{proof}

\begin{cor}\label{cor-sdr}
Assume that $X$ is projective over $C$ and that $K_X$ is
semi-ample over $C$. If $\cX$ is a good minimal $dlt$-model of
$X$, then $\Sk(\cX)$ is a strong deformation retract of
$X_K^{\an}$.
\end{cor}
\begin{proof}
Let $\cY\to \cX$ be a log resolution of $(\cX,\cX_s)$. By Theorem
\ref{thm-collapse}, the skeleton $\Sk(\cX)$ is a strong
deformation retract of $\Sk(\cY)$. By Theorem \ref{thm:deform},
$\Sk(\cY)$ is a strong deformation retract of $X_K^{\an}$.
\end{proof}

\subsection{Kontsevich-Soibelman skeleta}

\sss In \cite[\S4.5]{MuNi}, Musta\c{t}\u{a} and the first-named
author associated to every non-zero regular pluricanonical form
$\omega$ on $X_K$ a skeleton $\Sk(X_K,\omega)$ in $X_K^{\an}$,
generalizing a construction of Kontsevich and Soibelman
\cite{KoSo}.  The skeleton $\Sk(X_K,\omega)$ is
 precisely the locus of points of $X_K^{\an}$ where
 the weight function $\weight_{\omega}$ reaches its minimal value.
 If $\cX$ is any $snc$-model of $X$ over $\cC$, then
$\Sk(X_K,\omega)$ is a union of closed faces of $\Sk(\cX)$, which
can be explicitly computed \cite[4.5.5]{MuNi}. Taking the union of
the skeleta $\Sk(X_K,\omega)$ over all non-zero pluricanonical
forms $\omega$ on $X_K$, one obtains a topological subspace
$\Sk(X_K)$ of $X_K^{\an}$ that was called the essential skeleton
of $X_K$ in \cite[4.6.2]{MuNi}. It is an interesting birational
invariant of $X_K$. In this subsection, we will compare the
essential skeleton to the skeleton of a good minimal $dlt$-model
of $X$.

\begin{prop}\label{prop-KS} Assume that $K_X$ is semi-ample over $C$ and let $\cX$ be a $dlt$-model of $X$.
For every integer $m>0$ and every non-zero $m$-pluricanonical form
$\omega$ on $X_K$, we have
$$\Sk(X_K,\omega)\subset \Sk(\cX).$$
\end{prop}
\begin{proof}
  Let $x$ be a point
 of $\Sk(X_K,\omega)$. If $\red_{\cX}(x)$ is contained in
 $\cX^{\snc}$, then $x$ lies in
 $\widehat{\cX}_\eta$ and \cite[4.4.5]{MuNi} implies that $x$
 must lie in $\Sk(\cX)$, since the restriction of $\weight_{\omega}$ to $\widehat{\cX}_\eta$ can reach its minimal values only at points of $\Sk(\cX)$.

 Now suppose that
 $\red_{\cX}(x)$ is
 not contained in $\cX^{\snc}$. We will deduce a contradiction with the
 assumption that $x$ belongs to $\Sk(X_K,\omega)$.
 Replacing $\omega$ by its $d$-fold tensor power $\omega^{\otimes
d}$, with $d$ a positive integer, has no influence on the skeleton
$\Sk(X_K,\omega)$. Thus we may assume that the divisor
$$mK_{\cX}+m(\cX_s)_{\red}$$ is Cartier on $\cX$ and we denote by
$\mathcal{L}$ the associated line bundle. We choose a local
generator $\theta$ of $\mathcal{L}$ at the point $\red_{\cX}(x)$.
Note that the pullback of $\mathcal{L}$ to the regular locus
$\cX^{\mathrm{reg}}_R$ of $\cX_R$ is isomorphic to
$$\omega_{\cX^{\mathrm{reg}}_R/R}((\cX^{\mathrm{reg}}_s)_{\red})^{\otimes m}.$$
 We fix such an isomorphism. Then we can view $\omega$ as a
rational section of $\mathcal{L}$ and write $\omega=g\theta$
locally at $\red_{\cX}(x)$, with $g$ an element of
$$\mathcal{O}_{\cX_R,\red_{\cX}(x)}\otimes_R K.$$

Since $\cX$ is $\cC$-flat and normal, we know by \cite[7.4.1]{dJ} that for every affine open subscheme $\cU$ of $\cX$, the $R$-algebra of regular functions
 on the formal $t$-adic completion $\widehat{\cU}$ of $\cU_R$ coincides with the $R$-algebra of analytic functions on
 $\widehat{\cU}_\eta=\red_{\cX}^{-1}(\cU_s)$ whose absolute value is bounded by one. Thus by \cite[2.4.4]{berkbook}, there exists
 an irreducible component $E$ of $\cX_s$ whose
closure contains $\red_{\cX}(x)$ and such that the following property holds: if we denote by $\xi$ the generic point
of $E$ and by $x'$ the unique point in
$\red_{\cX}^{-1}(\xi)$, then $
 |g(x')|\geq |g(x)|$. But $x'$ belongs to $\Sk(\cX)$ (in the language of \cite{MuNi}, it is the divisorial point associated with the pair $(\cX^{\snc}_R,E\cap \cX^{\snc}_R)$), so that $\weight_{\omega}(x')=-\ln|g(x')|$.  It follows that
$$\weight_{\omega}(x')\leq -\ln |g(x)|.$$ Thus it is enough to show that
$$\weight_{\omega}(x)>-\ln|g(x)|$$ since in that case,
 $x$ cannot belong
 to the locus $\Sk(X_K,\omega)$ where $\weight_{\omega}$ reaches
its minimal value.

 Let $h:\cY\to \cX$ be a log-resolution of $(\cX,\cX_s)$. Then
 $\Sk(X_K,\omega)$ is contained in $\Sk(\cY)$.
 We denote by $\Delta$ the log pullback of $(\cX_s)_{\red}$ to
 $\cY$. Locally at $\red_{\cY}(x)$, it is explicitly given by
 $$\frac{1}{m}(\mathrm{div}(h^*g)-\mathrm{div}_{\cY}(\omega)).$$
 Since $\cX$ is a $dlt$-model and $\red_{\cX}(x)$ does not belong to $\cX^{\snc}$, we know that $\Delta<(\cY_s)_{\red}$ locally
 around $\red_{\cY}(x)$ by Lemma \ref{lemm:dlt}. Therefore,
  we can write
 \begin{eqnarray*} \weight_{\omega}(x)&=&v_x(\mathrm{div}_{\cY}(\omega)+m(\cY_s)_{\red})
 \\ &>&v_x(\mathrm{div}_{\cY}(\omega)+m\Delta)
 \\ &=&-\ln|g(x)|.
 \end{eqnarray*}
\end{proof}

\begin{theorem}\label{thm-ess}
If $K_X$ is semi-ample over $C$ and $\cX$ is a good minimal
$dlt$-model of $X$ over $\cC$, then
$$\Sk(X_K)=\Sk(\cX).$$ Moreover, if $m$ is a positive integer such
that $mK_{\cX}+m(\cX_s)_{\red}$ is Cartier and generated by global
sections $\omega_1,\ldots,\omega_r$ over some neighbourhood of $s$
in $\cC$, then
\begin{equation}\label{eq-KS} \Sk(X_K)=\bigcup_{i=1}^r \Sk(X_K,\omega_i).\end{equation}
\end{theorem}
\begin{proof} By Proposition \ref{prop-KS}, it is enough to show
that $\Sk(\cX)$ is contained in the right hand side of
\eqref{eq-KS}. Shrinking $\cC$ around $s$ if necessary, we can
assume that $mK_{\cX}+m(\cX_s)_{\red}$ is generated by global
sections $\omega_1,\ldots,\omega_r$. Then for each point $x$ on
$\cX^{\snc}_{s}$, we can choose an index $i$ in $\{1,\ldots,r\}$
 such that $\mathrm{div}_{\cX^{\snc}}(\omega_i)+m(\cX_s)_{\red}$ is an
effective divisor on $\cX$ and $x$ is not contained in its
support.
 This implies that the
weight $\weight_{\omega_i}$ of $\omega_i$ is zero at all points of
$\Sk(\cX)\cap \red_{\cX}^{-1}(x)$ and non-negative at all other
points of $X^{\an}_K$. Thus $\Sk(\cX)\cap \red_{\cX}^{-1}(x)$ is
contained in $\Sk(X_K,\omega_i)$. Varying the point $x$, we find
that $\Sk(\cX)$ is contained in $$\bigcup_{i=1}^r
\Sk(X_K,\omega_i).$$
\end{proof}

\begin{cor}\label{cor-defret}
If $X$ is projective over $C$ and $K_X$ is semi-ample over $C$,
then the essential skeleton $\Sk(X_K)$ is a strong deformation
retract of $X^{\an}_K$.
\end{cor}
\begin{proof}
 This follows from
 Theorem \ref{thm-mindlt}\eqref{item:exist},
 Corollary \ref{cor-sdr} and Theorem \ref{thm-ess}.
\end{proof}

\section{The essential skeleton of a Calabi-Yau
variety}\label{ss-CY}
\subsection{The skeleton is a pseudo-manifold}
\sss The case where $X$ is a
 family of Calabi-Yau varieties over $C$ is of particular
 interest; the connections with homological mirror symmetry were the main motivation for Kontsevich and
 Soibelman to define the skeleton in \cite{KoSo}. If
 $\omega$ is a volume form on $X_K$ (i.e., a
 nowhere vanishing differential form of maximal degree)
 then $\Sk(X_K)=\Sk(X_K,\omega)$ by \cite[4.6.4]{MuNi}. We will
 now prove that the underlying topological space of the essential skeleton $\Sk(X_K)$ is a pseudo-manifold with boundary; this
  result is implicitly contained in \cite{KK,Kol11}.

\sss A topological space $T$ is called an $n$-dimensional {\em
pseudo-manifold with boundary} if it admits a triangulation
$\mathscr{T}$ satisfying the following conditions:
\begin{enumerate}
\item  \label{it:equi} (dimensional homogeneity)  $T =
|\mathscr{T}|$ is the union of all $n$-simplices.
\item
\label{it:pseudo} (non-branching) Every $(n - 1)$-simplex is a
face of precisely one or two $n$-simplices.
\item \label{it:connect} (strong
connectedness) \label{it:strong} For every pair of $n$-simplices
$\sigma$ and $\sigma'$ in $\mathscr{T}$, there is a sequence of
$n$-simplices
$$\sigma = \sigma_0, \sigma_1, \ldots, \sigma_\ell = \sigma'$$ such
that the intersection $\sigma_i \cap \sigma_{i+1}$ is an $(n -
1)$-simplex for all $i$.
\end{enumerate}
 We say that $T$ is a {\em closed pseudo-manifold} if we can replace condition \eqref{it:pseudo} by the property that
  every $(n - 1)$-simplex is a
face of precisely two $n$-simplices. A typical example of a
2-dimensional closed pseudo-manifold which is not a manifold is
the pinched torus.

\sss For the reader's convenience, we include some basic facts
about adjunction for $dlt$-pairs. We refer to Chapter 4 of
\cite{Kol13} for more background.  Let $(Y,\Delta)$ be a
$dlt$-pair over $k$, and let $D$ be a log canonical center of
$(Y,\Delta)$. Then $D$ is normal, by \cite[4.16]{Kol13}. There is
a well defined $\mathbb{Q}$-divisor $\Delta_D$ on $D$, called the
{\em different}
 of $\Delta$ on $D$ \cite[4.18]{Kol13}, which is induced by the Poincar\'e map and
satisfies the equation
$$(K_Y+\Delta)|_{D}=K_D+\Delta_D.$$
 In the sequel, whenever we write such an equation it will be
 understood that $\Delta_D$ is the different of $\Delta$ on $D$.
 The pair $(D,\Delta_D)$ is
again a $dlt$-pair, by \cite[4.19]{Kol13}. Write $\lfloor
\Delta\rfloor=\sum_{i\in I}D_i$. If $J$ is a subset of $I$ and $D$
is a component of $\bigcap_{j\in J}{\Delta}_j$,
 it is not hard to see that for every non-empty subset $J'$ of $I\setminus J$,
   every irreducible component of the intersection
$$D\cap \bigcap_{j\in J'} \Delta_{j}$$
 is a log canonical center of $(D,\Delta_D)$ (see
\cite[4.19]{Kol13}).  Conversely, by repeatedly using inversion of
adjunction \cite[4.9]{Kol13}, one sees that any log canonical
center of $(D,\Delta_D)$ is a log canonical center of
$(Y,\Delta)$, and thus an irreducible component of an intersection
$D\cap \bigcap_{j\in J'} \Delta_{j} $ for some non-empty subset
$J'$ of $I\setminus J$.

\begin{theorem}\label{thm:CY}
 Assume that $X$ is projective over $C$ and $K_X$ is $\Q$-linearly equivalent to $0$ over $C$. Then the underlying topological space of $\Sk(X_K)$ is
a pseudo-manifold with boundary.
\end{theorem}
\begin{proof}
As we mentioned above, this result is essentially contained in
\cite{KK,Kol11}.  Using the terminology there, properties (1)-(3)
of a pseudo-manifold all follow from the fact that two minimal log
canonical centers of a log crepant structure are
$\mathbb{P}^1$-linked in the sense of Definition 9 in
\cite{Kol11}. We will now explain this in more detail. We denote
by $n$ the relative dimension of $X$ over $C$.

 By Theorem
\ref{thm-mindlt}\eqref{item:exist}, there exists a a good minimal
$dlt$-model $\cX$ of $X$ over $\cC$.  By Theorem \ref{thm-ess}, we
have $\Sk(X_K)=\Sk(\cX)$. As a triangulation on $\Sk(\cX)$, we
take the first barycentric subdivision of the simplicial structure
on $\Sk(\cX)$. This barycentric subdivision is necessary to
guarantee that the intersection of two faces is a codimension one
face of both, rather than a union of faces (think of a type
$I_{2}$ degeneration of elliptic curves, whose skeleton consists
of two vertices joined by two edges).


We choose an integer $m>0$ such that $mK_X\sim 0$. Since the
divisor
 $mK_{\cX}+m(\cX_s)_{\red}$ is semi-ample over $\cC$ and trivial over $C$, we see that
$mK_{\cX}+m(\cX_s)_{\red}$ must be a multiple of $\cX_s$ and thus
trivial over $\cC$. Thus we can apply  Theorem 10 in \cite{Kol11}
to the $dlt$-pair $(\cX,(\cX_s)_{\red})$ over $\cC$.
 It states that
every two minimal log canonical centers $D$ and $D^*$ of
$(\cX,(\cX_s)_{\red})$ are $\mathbb{P}^1$-linked.  This  means, in
particular, that they have the same dimension, say $n-d$, and that
there exist a sequence of $(n-d+1)$-dimensional log canonical
centers $E_1,E_2,\ldots,E_\ell$ and a sequence of
$(n-d)$-dimensional log canonical centers $D=D_0,D_1,\ldots,
D_\ell=D^*$ such that $D_{i-1},D_{i}\subset E_i$ for $1\leq i\leq
\ell$.
  In this way, we obtain properties \eqref{it:equi}
 and \eqref{it:connect} of a pseudo-manifold with boundary.

 If we
have two minimal log canonical centers $D_1,\,D_2$ of
$(\cX,(\cX_s)_{\red})$, contained in an $(n-d+1)$-dimensional
 log canonical center $E$, and if we write
 $$(K_{\cX}+ (\cX_s)_{\red})|_{E}=K_E+D_1+D_2+\Delta$$
 for some $\Delta\ge 0$, then $(E,D_1+D_2+\Delta)$ is again a $dlt$-pair
 \cite[4.19]{Kol13}. Moreover, $D_1$ cannot intersect $D_2$ or
 $\lfloor \Delta \rfloor$ because the intersection would be a
 union of log canonical centers of $(\cX,(\cX_s)_{\red})$, which
 contradicts the minimality of $D_1$. Thus we are in the
 situation of the second part of the proof of Theorem 10 in
 \cite{Kol11}.
 That proof shows that $D_1$ and $D_2$ are the only log canonical centers of $(E,D_1+D_2+\Delta)$.  Property \eqref{it:pseudo} follows.
  \end{proof}

\sss We can say more in the case where $\Sk(\cX)$ has maximal
dimension, that is, dimension equal to $n=\dim (X_K)$. First, we
need a lemma.
\begin{lem}\label{l-adj}
Let $(Z,\Delta=\sum^j_{i=1}\Delta_i)$ be a reduced $dlt$-pair
 over $k$ such that $K_Z+\Delta$ is Cartier. Let $D$ be a log canonical
center and let $\Delta_1,\ldots,\Delta_\ell$ be the irreducible
components of $\Delta$ that contain $D$. Let $U$ be the maximal
open subset of $Z$ where $Z$ is smooth and $\Delta$ is
 a divisor with strict normal crossings. If we write $(K_Z+\Delta)|_D=K_D+\Delta_D$,
then $\Delta_D$ is equal to the closure of the restriction of
$\sum^j_{i=\ell+1}\Delta_i|_U$ to $U\cap D$.
\end{lem}
\begin{proof}We first notice that  in the above statement,
$U$ can be replaced by any smaller open set that meets all the log
canonical centers: the closure of the restriction of
$\sum^j_{i=\ell+1}\Delta_i|_U$ to $U\cap D$ will yield the same
divisor on $D$.

Then by induction,  we only need to treat the case where $D$ is a
component of $\Delta$, say $\Delta_1$. If we take a log resolution
$f:Y\to (Z,\Delta)$ and let $D'=\Delta'_1$ be the birational
transform of $D$, then $\Delta_D$ can be computed as follows: if
we write $f^*(K_Z+\Delta)|_{D'}=K_{D'}+\Delta_{D'}$ then
$\Delta_D=(f|_{D'})_*(\Delta_{D'})$. In particular, as
$K_Z+\Delta$ is Cartier, we know that $\Delta_D$ is a integral
divisor. Since it is effective, and all the components of
$\Delta_D$ are log canonical centers of $(X,\Delta)$, we see
 that $\Delta_D$ must be equal to the closure of  the restriction of
$\sum^j_{i=2}\Delta_i|_U$.
\end{proof}

\begin{theorem}\label{thm:noboundary}
We suppose that $k$ is algebraically closed. Assume that $X$ is
projective over $C$, $K_X$ is trivial over $C$ and $X$ has a
projective
 $snc$-model $\cY$ over $\cC$ with reduced special fiber $\cY_s$. Assume moreover that $\Sk(X_K)$ is of dimension
$n=\dim(X_K)$. Then $\Sk(X_K)$ is an $n$-dimensional closed
pseudo-manifold.
\end{theorem}
\begin{proof} By running MMP for $\cY$ over $\cC$, we know that $X$ has a good
minimal dlt model $\cX$ with reduced special fiber (see
\cite{Fujino} or \cite{HX}).
 Then one sees as in the proof of Theorem \ref{thm:CY} that $K_{\cX}$ is trivial over $\cC$. Our
assumption on the dimension of $\Sk(X_K)$ implies that the minimal
log canonical centers of $(\cX, \cX_s)$ are points. Let $D$ be a
one-dimensional log canonical center, and let $D_i$ ($1\le i \le
\ell$) be the 0-dimensional log canonical centers contained in
$D$. From Lemma \ref{l-adj}, we know that
$$(K_{\cX}+\cX_s)|_{D}=K_{D}+\sum^\ell_{i=1}
D_i\sim 0.$$ Thus $D$ is a rational curve and $\ell=2$, which
means that $\Sk(X_K)$ is closed.
\end{proof}

\sss We do not know any example where $X$ satisfies the conditions of
 Theorem \ref{thm:noboundary} and $\Sk(X)$ is not a topological manifold (or even a sphere).
 Even if we omit the condition that the skeleton has maximal dimension, we do not know an example where $X$ is not
 a topological manifold with boundary.
 In Theorem \ref{thm:noboundary}, the condition that $\Sk(X)$
has maximal dimension cannot be omitted; for instance, there are
examples of semi-stable degenerations of K3-surfaces with trivial
relative canonical sheaf where the special fiber is a chain of
surfaces, so that the skeleton is homeomorphic to a closed
interval. We will now give an interpretation of this condition in
terms of the monodromy around $s\in \cC$.

\begin{lemma}\label{lemm:bc}
Let $Y$ be a connected smooth and proper $K$-variety and let
$\omega$ be a non-zero $m$-pluricanonical form on $Y$, for some
$m>0$. Let $K'$ be a finite extension of $K$, set $Y'=Y\times_K
K'$ and denote by $\omega'$ the pullback of $\omega$ to $Y'$. Then
 the skeleton $\Sk(Y',\omega')$ is the inverse image of
 $\Sk(Y,\omega)$ under the projection morphism $\pi:(Y')^{\an}\to
 Y^{\an}$. In particular, $\Sk(Y,\omega)$ and $\Sk(Y',\omega')$ have the same dimension.
\end{lemma}
\begin{proof}
 Since $\pi$ is surjective, we can reduce to the case where $K'$ is Galois over $K$ by replacing $K'$ by a Galois closure over $K$. Let $d$ be the ramification index of $K'$ over $K$. We will prove that
 $$\weight_{\omega'}(y)=d\cdot \weight_{\omega}(\pi(y))+(1-d)m$$ for every
 divisorial point $y$ on $(Y')^{\an}$ (see \cite[2.4.10]{MuNi} for the notion of divisorial point). This immediately implies the
 first statement in the lemma, since
 $\Sk(Y,\omega)$ is the closure of the set of divisorial points where the weight function reaches its minimal value \cite[4.5.1]{MuNi}. The equality of dimensions then follows from the fact that $\pi$ has finite fibers.

 We denote by $R'$ the integral closure of $R$ in $K'$. Let $\cY'$ be a regular separated $R'$-scheme of finite type with
 irreducible special fiber $\cY'_k$, endowed with an isomorphism
 of $K'$-schemes
 $\cY'_{K'}\to Y'$. Let $y$ be the unique point in
 $\red_{\cY'}^{-1}(\xi)$, where $\xi$ denotes the generic point of
 $\cY'_k$. Removing a closed subset of $\cY'_k$ if necessary, we
 can find a regular separated $R$-scheme of finite type $\cY$ and
 an isomorphism $\cY_K\to Y$ such that $\cY'$ is an open subscheme
 of the normalization $\cZ$ of $\cY\times_R R'$. Then $\red_{\cY}(\pi(y))$ is a generic point of
 $\cY_k$.

  If we use the notations from \eqref{sss:log} and denote by $(S')^+$ the log scheme associated to $R'\setminus \{0\}\to R'$, then
  the $(S')^+$-log scheme
  $(\cY')^+$ is isomorphic to an open log subscheme of the $fs$ (fine and saturated) base change $\cZ^+$ of
  $\cY^+$ from $S^+$ to $(S')^+$, because the underlying scheme of this $fs$ base change is precisely $\cZ$.
  Since log differentials are
  compatible with $fs$ base change, the divisor $\mathrm{div}_{\cZ^+}(\omega')$ from \eqref{sss:log} is the pullback of
  $\mathrm{div}_{\cY^+}(\omega)$ to $\cZ$.
   Thus  the
  description of the weight function in \eqref{sss:log} yields
 \begin{eqnarray*}
 \weight_{\omega'}(y)-m&=&v_y(\mathrm{div}_{\cZ^+}(\omega'))
 \\&=& d\cdot v_{\pi(y)}(\mathrm{div}_{\cY^+}(\omega))
 \\ &=&d\cdot \weight_{\omega}(\pi(y))-dm
 \end{eqnarray*}
  (the scaling factor $d$ is caused by the renormalization of the
 discrete valuation on $K'$).
\end{proof}

\begin{theorem}\label{thm:sphere}
Assume that $k=\C$ and denote by $n$ the relative dimension of $X$
over $C$. Suppose that $X$ is projective over $C$ and that
 $K_X$ is trivial over $C$. Let $F$ be a general fiber of the morphism $X\to C$. Then
 $\Sk(X_K)$ has dimension $n$ if and only if the monodromy transformation around $s\in \cC$ on $H^n(F(\C),\Q)$ has a Jordan
 block of size $n+1$. If this holds, $X$ has a projective $snc$-model $\cY$ over $\cC$ such that
 $\cY_s$ is reduced, and $h^{i,0}(F)=0$ for $0<i<n$,
   then $\Sk(X_K)$ has the $\Q$-homology
 of an $n$-sphere.
\end{theorem}
\begin{proof}
  By Lemma \ref{lemm:bc} and the Semi-Stable Reduction Theorem we
 can assume that $X$ has a projective $snc$-model $\cY$ over $\cC$ such that
 $\cY_s$ is reduced, by replacing $K$ by a suitable finite extension. For every integer $i\geq 0$, we denote by $$\mathbf{H}^i=\mathbb{H}^i(\cY_s,R\psi_{\cY}(\Z))\cong H^i(F(\C),\Z)$$
  the degree $i$ nearby cohomology of $\cY$ at
 $s$; here $R\psi_{\cY}(\Z)$ denotes the complex of nearby cycles
 with $\Z$-coefficients associated to $\cY$. By \cite{steenbrink},
 the spaces $\mathbf{H}^i$ carry a canonical mixed Hodge structure,
 whose weight filtration coincides with the monodromy filtration.
  In particular, there exists a Jordan
 block of monodromy of size $n+1$ on $\mathbf{H}^n_{\Q}$ if and only if
 $W_0\mathbf{H}^n_{\Q}\neq 0$.

  By \cite[5.1]{berk-limit} and its proof, the $\Q$-vector space $W_0\mathbf{H}^i_{\Q}$ is canonically
   isomorphic to the degree $i$ singular
  cohomology of $X_K^{\an}$, for every $i\geq 0$. Since $X_K^{\an}$ is homotopy
  equivalent to $\Sk(X_K)$ by Corollary \ref{cor-defret}, we see
  that $W_0\mathbf{H}^n_{\Q}$ can only be different from zero if the
  dimension of $\Sk(X_K)$ is equal to $n$.
   We will now prove the
  converse implication. Suppose that $\Sk(X_K)$ has dimension $n$
  and let $\omega$ be a relative volume form on $X$ over $C$ such
  that $\omega$ extends to a global section of $\omega_{\cY/\cC}(\log \cY_s)$
   that generates $\omega_{\cY/\cC}(\log \cY_s)$ at at least one generic point
   of $\cY_s$ (modulo shrinking $\cC$, such $\omega$ always
   exists). Then it follows from \cite[4.5.5]{MuNi} that $\Sk(X_K)$ is the simplicial subspace of
   $\Sk(\cY)$ spanned by the vertices corresponding to the
   irreducible components $E$ of $\cY_s$ such that $\omega$
   generates $\omega_{\cY/\cC}(\log \cY_s)$ at the generic point of $E$. Since
   $\Sk(X_K)$ has dimension $n$, we can find such components
   $E_1,\ldots,E_n$ that intersect in a point. Denote by $D$ the union of $n$-fold intersection points of components of $\cY_s$.
     Then by reduction modulo $t$, $\omega$ induces an element of
 $$H^0(\cY_s,\omega_{\cY/\cC}(\log \cY_s)\otimes
   \mathcal{O}_{\cY_s})$$
   whose image under the
   Poincar\'e residue map
$$\mathcal{R}:H^0(\cY_s,\omega_{\cY/\cC}(\log \cY_s)\otimes
   \mathcal{O}_{\cY_s})\to H^0(\cY_s,\mathrm{Gr}_{-n}^W(\omega_{\cY/\cC}(\log \cY_s)\otimes
   \mathcal{O}_{\cY_s}))\cong H^0(D,\mathcal{O}_D)$$
   is different
   from zero.
 However, by the degeneration of the Hodge and weight spectral sequences,
  the image of $\mathcal{R}$ injects into
  $W_0\mathbf{H}^n_{\C}$.
 Thus $W_0\mathbf{H}^n_{\Q}$ is
   non-trivial.

 Finally, assume that $\Sk(X_K)$ has dimension $n$ and that $h^{i,0}(F)=0$ for $0<i<n$.
 Then $$\mathrm{Gr}_F^0\mathbf{H}^i_{\C}\cong H^i(\cY_s,\mathcal{O}_{\cY_s})=0$$ for $0<i<n$ and
 $$\mathrm{Gr}_F^0\mathbf{H}^i_{\C}\cong H^i(\cY_s,\mathcal{O}_{\cY_s})\cong \C$$ for $i=0,n$ by the degeneration
 of the Hodge spectral sequence for the limit mixed Hodge structure.
  Thus
  $W_0\mathbf{H}^i_{\Q}=0$ for $0<i<n$, $W_0\mathbf{H}^0_{\Q}\cong \Q$ and $W_0\mathbf{H}^n_{\Q}$
  has dimension at most one; it must have dimension one since we
  have already proven that it is non-zero.
 It follows that $\Sk(X_K)$ has the $\Q$-homology of an $n$-sphere.
\end{proof}

\subsection{Removing the algebraicity condition}
\sss In this section, we will extend Theorems \ref{thm-collapse},
 \ref{thm:CY}, \ref{thm:noboundary} and \ref{thm:sphere} to the case where $X$ is a Calabi-Yau variety over $K=k\llpar
t\rrpar $ instead of over the curve $C$.  The crucial point is
that the skeleton of a Calabi-Yau variety can be computed from the
logarithmic structure on the special fiber of any $snc$-model.

\begin{prop}\label{prop:indep}
Let $\cY$ be a connected regular flat proper $R$-scheme such that
$\cY_k$ is a strict normal crossings divisor.
  Then for every
connected
 flat proper $R$-scheme $\cZ$ and every isomorphism of $R/(t^2)$-schemes $$f:\cY\times_R R/(t^2)\to \cZ\times_R R/(t^2),$$
the following properties hold.
\begin{enumerate}
\item \label{it:reg} The scheme $\cZ$ is regular and $\cZ_k$ is a
divisor with strict normal crossings.
 \item \label{it:log} Denote by $S^+$ the log scheme associated to $R\setminus \{0\}\to R$ and by $\cY^+$ and $\cZ^+$ the
  schemes $\cY$ and $\cZ$ endowed with the divisorial log structures associated to their special fibers. For every integer $d>0$ we denote
  by $s^+_d$ the standard log point $(\Spec k,k^*\oplus \N)$ viewed as a log scheme over $S^+$ {\em via} the morphism of charts $\N\to \N:n\mapsto dn$.
   If we denote by $e$ the least common multiple of the multiplicities of the components of $\cY_k$, then
  there exists an isomorphism of log schemes $$g:\cY^+\times_{S^+}s^+_e\to \cZ^+\times_{S^+}s^+_e$$
 over $s^+_e$, such that
 $g$ is compatible with the reduction of $f$ modulo $t$ (meaning that the obvious square in the category of $k$-schemes commutes).
\end{enumerate}
\end{prop}
\begin{proof}
 It is easy to see that \eqref{it:reg} holds,
 since we can detect regularity by looking at the dimensions of the Zariski tangent spaces
 at the points of $$\cY\times_R R/(t^2)\cong \cZ\times_R
 R/(t^2).$$ Moreover, the special fibers of $\cY$ and $\cZ$ are
 isomorphic so that $\cZ_k$ is a divisor with strict normal
 crossings.
 Point \eqref{it:log} is more subtle and follows from
 \cite[2.6(2)]{kisin}.
\end{proof}

\begin{prop}\label{prop:sklog}
Let $\cY$ be a connected regular flat proper $R$-scheme such that
$\cY_K$ has trivial canonical sheaf and $\cY_k$ is a strict normal
crossings divisor. Then the skeleta $\Sk(\cY)$ and $\Sk(\cY_K)$
only depend on $\cY\times_R R/(t^2),$
 in the following sense. Assume that $\cZ$ is a regular flat proper $R$-scheme such that $\cZ_K$ has trivial canonical sheaf and there exists an
 isomorphism of $R/(t^2)$-schemes $$f:\cY\times_R R/(t^2)\to
 \cZ\times_R R/(t^2).$$
 Then there exists an isomorphism of simplicial
spaces $\Sk(\cY)\to \Sk(\cZ)$ that maps $\Sk(\cY_K)$ onto
$\Sk(\cZ_K)$.
\end{prop}
\begin{proof}
  Reducing $f$ modulo $t$, we obtain an isomorphism of $k$-schemes
 $\cY_k\to \cZ_k$ and, by taking the dual intersection complexes, an
 isomorphism of simplicial spaces with piecewise $\Z$-affine structure $\Sk(\cY)\to \Sk(\cZ)$. We will
 prove that this isomorphism maps $\Sk(\cY_K)$ onto $\Sk(\cZ_K)$.

 We use the notations from Proposition
 \ref{prop:indep}\eqref{it:log} and we set $s^+=s_1^+$. We denote by $\cY^+_k$ the
 log scheme $\cY^+\times_{S^+} s^+$ obtained by
 restricting the log structure on $\cY^+$ to the special fiber
 $\cY_k$ of $\cY$.
  It follows from
   \cite[7.1]{IKN} that
 $$\Omega:=H^0(\cY,\omega_{\cY^+/S^+})$$ is a free $R$-module of rank one
 and that the reduction map
 $$\Omega\otimes_R k\to \Omega_k:=H^0(\cY_k,\omega_{\cY^+_k/s^+})$$ is an
 isomorphism. Let $\omega$ be a generator of the $R$-module $\Omega$
 and denote by $\omega_k$ its image in $\Omega_k$. By \eqref{sss:log}, the
 generic point $\xi$ of an
 irreducible component $E$ of $\cY_k$ is $\omega$-essential in the
 sense of \cite[4.5.4]{MuNi} if and only if $\omega_k$ generates
 $\omega_{\cY^+_k/s^+}$ at the point $\xi$. Moreover, the skeleton $\Sk(\cY_K)=\Sk(\cY_K,\omega)$ is
 the simplicial subspace of $\Sk(\cY)$ spanned by the vertices
 corresponding to such points $\xi$ \cite[4.5.5]{MuNi}. However,
 for every integer $d>0$, the stalk of $\omega_{\cY^+_k/s^+}$ at $\xi$ is generated by
 global sections if and only if $\omega_{\cY^+\times_{S^+}s_d^+/s_d^+}$ is generated by global sections
 at any point lying above $\xi$, by the base change property in \cite[7.1]{IKN}. The analogous statements hold for
 $\cZ$. Thus it follows from Proposition \ref{prop:indep}\eqref{it:log} that the isomorphism $\Sk(\cY)\to \Sk(\cZ)$ maps $\Sk(\cY_K)$ onto $\Sk(\cZ_K)$.
\end{proof}

\begin{theorem}\label{thm-form}
Let $X$ be a geometrically connected, smooth and projective
 $K$-variety with trivial canonical sheaf. Then the following
properties hold.
\begin{enumerate}\item\label{it:strongdef2} The essential skeleton $\Sk(X)$ is a strong deformation retract of
$X^{\an}$. \item \label{it:collapse2} If $\cX$ is a
 projective $snc$-model of $X$ over $R$, then $\Sk(X)$ is contained
in $\Sk(\cX)$ and can be obtained from $\Sk(\cX)$ (as a
topological subspace of $\Sk(\cX)$ with piecewise affine
structure) by a finite number of elementary collapses.
 \item \label{it:pseudo2} The essential skeleton $\Sk(X)$ is a
 pseudo-manifold with boundary. If $k$ is algebraically closed and $\Sk(X)$ has dimension
 $\mathrm{dim}(X)$, then it is a closed pseudo-manifold.
 \item \label{it:sphere} Assume that $k$ is algebraically closed. Let $\sigma$ be a topological generator of the absolute Galois group $G(K^a/K)$ and let $\ell$ be a prime.
  Then
 $\Sk(X_K)$ has dimension $n=\dim(X)$ if and only if the action of $\sigma$ on
 $$H^n_{\mathrm{\acute{e}t}}(X\times_K K^a,\Q_\ell)$$ has a Jordan
 block of  size $n+1$. If this holds, $X$ has a projective $snc$-model over $R$ with reduced special fiber, and $h^{i,0}(X)=0$ for $0<i<n$,
   then $\Sk(X_K)$ has the
  $\Q$-homology of an $n$-sphere.
\end{enumerate}
\end{theorem}
\begin{proof}
 Let $\cX$ be a projective $snc$-model of $X$ over $R$.
    By a standard argument based on spreading out and Greenberg Approximation (as explained in \cite[5.1.2]{MuNi}, for instance) we can find a connected smooth $k$-curve
 $\cC$, a $k$-rational point $s$ on $\cC$, a uniformizer $t$ in $\mathcal{O}_{\cC,s}$ and a smooth and
 projective
  $\cC$-scheme $\cX'$ with geometrically connected fibers such that there exists an isomorphism
 $$\cX\times_R R/(t^2)\to \cX'\times_{\cC} \Spec
 \mathcal{O}_{\cC,s}/(t^2)$$ over $R/(t^2)\cong {O}_{\cC,s}/(t^2)$. Inspecting the proof of \cite[5.1.2]{MuNi}, we
 see that we can also assume that the relative canonical sheaf of  $\cX'$ is trivial over $C=\cC\setminus \{s\}$ (if the generic fiber of a smooth and proper family over
  an integral scheme
  has trivial canonical sheaf, then this holds for
  all fibers over some dense open subscheme of the base).

  By \eqref{sss:defret} we know that $\Sk(\cX)$ is a strong deformation retract of $X_K^{\an}$. Thus by Proposition
  \ref{prop:sklog}, it suffices to prove assertions \eqref{it:strongdef2}--\eqref{it:pseudo2} for $\cX'\times_{\cC}\Spec K$ instead of $X$. In this
  case, they follow from Corollary \ref{cor-defret} and
 Theorems
 \ref{thm-collapse}, \ref{thm-ess}, \ref{thm:CY} and \ref{thm:noboundary}.

 Invoking the Lefschetz Principle, we may assume that $k=\C$. By Proposition \ref{prop:indep} and the
  theory of logarithmic nearby cycles \cite[3.3]{nakayama} the action of $\sigma$ on
 $$H^n_{\mathrm{\acute{e}t}}(X\times_K K^a,\Q_\ell)$$ has a Jordan
 block of  size $n+1$ if and only if the corresponding statement
 holds for $\cX'_K$. By Deligne's comparison theorem for \'etale
 and complex analytic nearby cycles in \cite[Exp.XIV]{sga7b}, it
 is also equivalent to the property that the monodromy action on
 the degree $n$ singular cohomology of a general fiber of $\cX'$
 has a Jordan block of size $n+1$. If $h^{i,0}(X)=0$ for $0<i<n$,
 then we can assume that this also holds for a general fiber of
 $\cX'$, by the proof of \cite[5.1.2]{MuNi} (if this property is
 satisfied by the generic fiber of a smooth and proper family over
 an integral scheme, then it holds for all fibers over a dense
 open subscheme of the base, by semi-continuity).
  Thus the assertion \eqref{it:sphere} follows from Theorem
 \ref{thm:sphere}.
\end{proof}

\end{document}